\documentclass[11pt]{article}
\usepackage{authblk}
\usepackage{amssymb}
\usepackage[all]{xy}
\usepackage{mathrsfs}
\usepackage{amsmath,amssymb,amsthm}
\usepackage{extarrows}
\usepackage{mathrsfs}
\newtheorem{theorem}{Theorem}[section]
\newtheorem{lemma}[theorem]{Lemma}
\newtheorem{proposition}[theorem]{Proposition}
\newtheorem{definition}[theorem]{Definition}

\newtheorem{corollary}[theorem]{Corollary}
\newtheorem{example}[theorem]{Example}

\def\to{\rightarrow}

\def\f{\mathfrak}

\def\c{\mathcal}

\def\r{\mathrm}
\def\bb{\mathbb}

\def\ov{\overline}
\date{}
\begin{document}
\title{Topological Entropy for Arbitrary Subsets of\\ Infinite Product Spaces}
\author{Maysam Maysami Sadr\thanks{sadr@iasbs.ac.ir}\hspace{2.5mm}\&\hspace{1mm}Mina Shahrestani\thanks{m.shahrestani@iasbs.ac.ir}}
\affil{Department of Mathematics, Institute for Advanced Studies in Basic Sciences, Zanjan, Iran}
\maketitle
\begin{abstract}
In this note a notion of generalized topological entropy for arbitrary subsets of the space of all sequences
in a compact topological space is introduced.
It is shown that for a continuous map on a compact space the generalized topological entropy of the set of all orbits of the map coincides
with the classical topological entropy of the map. Some basic properties of this new notion of entropy are considered; among them are:
the behavior of the entropy with respect to disjoint union, cartesian product, component restriction and dilation, shift mapping,
and some continuity properties with respect to Vietoris topology. As an example, it is shown that any self-similar structure of a fractal
given by a finite family of contractions gives rise to a notion of intrinsic topological entropy for subsets of the fractal.
A generalized notion of Bowen's entropy associated to any increasing sequence of compatible semimetrics on a topological space is introduced and some
of its basic properties are considered. As a special case for $1\leq p\leq\infty$ the Bowen $p$-entropy of sets of sequences of any metric space
is introduced. It is shown that the notions of generalized topological entropy and Bowen $\infty$-entropy for compact metric spaces coincide.

\textbf{MSC 2020.} 37B40, 54C70.

\textbf{Keywords.} topological entropy, Bowen entropy, space of infinite-sequences of a compact space.
\end{abstract}
\section{Introduction}\label{2005240804}
The notion of topological entropy for continuous maps on compact spaces have been introduced by Adler, Konheim, and
McAndrew in \cite{AdlerKonheimMcAndrew1} as a topological version of the measure theoretic entropy in
Ergodic Theory defined by Kolmogorov and Sinai \cite{Sinai1}.
Then Bowen \cite{Bowen1} defined a notion of entropy for uniformly continuous maps on arbitrary metric spaces and showed \cite{Bowen2}
that his notion coincides with the topological entropy on compact metric spaces. Also, Bowen's definition has been extended to uniformly
continuous maps on uniform spaces by Hood \cite{Hood1}. For more details on the subject and various notions of
entropy we refer the reader to \cite{Downarowicz1}.

The main aim of this note is to introduce a new generalized notion of topological entropy. Indeed we show that for any compact topological space
$X$ there is a \emph{well-behaved} and \emph{intrinsic} notion of entropy for arbitrary subsets of the infinite product space $X^\infty$.
It will be clear that for any continuous map $T:X\to X$ if $\ov{T}$ denotes the subset of $X^\infty$ of all orbits of $T$ then our entropy value $\r{H}(\ov{T})$ coincides with the classical topological entropy of $T$. This suggests that for an arbitrary noncontinuous map or set-valued
map $T$ we consider the value $\r{H}(\ov{T})$ as topological entropy of $T$.

In Section \ref{2004270842} we consider the main definition of this note. In Section \ref{2005240803} we consider some basic properties of
our notion of entropy. Among them are: the behavior of the entropy with respect to disjoint union, cartesian product,
component restriction and dilation, shift mapping, and some continuity properties with respect to Vietoris topology.
We show that for any compact Hausdorff space (finite or infinite) any real number $u$ with $0\leq u\leq\log|X|$ appears as the entropy of a
subset of $X^\infty$. We also consider some examples and show that any self-similar structure of a fractal given by a finite family of
contractions gives rise to a notion of intrinsic topological entropy for subsets of the fractal. In Section \ref{2005300733} we consider
a Bowen's notion of entropy for arbitrary subsets of any topological space endowed with an increasing sequence of compatible semimetrics
and we consider some elementary properties of this new notion. In Section \ref{2005240802} as an special case of the notion given 
in Section \ref{2005300733} we introduce a notion of Bowen $p$-entropy ($1\leq p\leq\infty$) for subsets of $X^\infty$ where $X$
is an arbitrary metric space. It will be clear that the Bowen $\infty$-entropy of $\ov{T}$ for a continuous map $T:X\to X$ 
coincides with the usual Bowen entropy of $T$. We also show that Bowen $\infty$-entropy of closed subsets of $X^\infty$ coincides with
the generalized topological entropy given in Section \ref{2004270842} when $X$ is a compact metric space.

\textbf{Notations.}
For a topological space $X$ we denote by $\ov{X}:=X^\infty$ the space of all sequences $(x_n)_{n\geq0}$ of elements of $X$
endowed with the product topology.
\section{The Main Definition}\label{2004270842}
Let us begin by some conventions. For open covers $\c{U},\c{V}$ of a topological space $X$ we write $\c{U}\preccurlyeq\c{V}$
if every member of $\c{V}$ is a subset of a member of $\c{U}$. $\bigvee_{i=1}^n\c{U}_i$ denotes the \emph{join}
of the finite family $\c{U}_1,\ldots,\c{U}_n$ of open covers of $X$ i.e. $\bigvee_{i=1}^n\c{U}_i:=\{\cap_{i=1}^n U_i:U_i\in\c{U}_i\}$.
If $\c{V}_i$ is an open cover of a space $X_i$ for $i=1,\ldots,n$ then the \emph{associated product cover} is defined to be the
open cover of $\prod_iX_i$ given by $\{\prod V_i:V_i\in\c{V}_i\}$.

Let $S$ be a nonempty subset of a space $X$. For an open cover $\c{U}=\{U_i\}_{i\in I}$
of $X$ we denote by $\r{N}_{S/X}(\c{U})$ the minimum of the cardinals $|J|$ where $J\subseteq I$ is such that $S\subseteq\cup_{i\in J}U_i$.
We denote the value $\r{N}_{X/X}(\c{U})$ simply by $\r{N}(\c{U})$.

Topological entropy of continuous mappings \cite{Walters1} is defined as follows:
Let $T:X\to X$ be a continuous mapping. For any $n\geq0$ denote by $T^{-n}\c{U}$ the open cover of $X$ given by $\{T^{-n}(U_i)\}_{i\in I}$. Let
\begin{equation}\label{2003311300}
\r{h}(T,\c{U}):=\limsup_{n\to\infty}\frac{1}{n}\log\r{N}(\vee_{i=0}^{n-1}T^{-i}\c{U}).
\end{equation}
It is well-known that the $\limsup$ in (\ref{2003311300}) is a limit and is equal to $\inf_n\frac{1}{n}\log\r{N}(\vee_{i=0}^{n-1}T^{-i}\c{U})$.
If $X$ is compact then the \emph{topological entropy of $T$} is defined to be the value
\begin{equation}\label{2003311359}
\r{h}(T):=\sup_\c{U}\r{h}(T,\c{U})
\end{equation}
where the supremum is taken over all open covers $\c{U}$ of $X$.

We generalize the classical definitions given by (\ref{2003311300}) and (\ref{2003311359}) as follows.
For an open cover $\c{U}=\{U_i\}_{i\in I}$ of $X$ and any integer $n\geq1$ let
$\ov{\c{U}}^n$ denote the open cover of $\ov{X}$ given by
$$\ov{\c{U}}^n:=\Big\{U_{i_0}\times\cdots\times U_{i_{n-1}}\times X\times X\times\cdots\big|i_0,\ldots,i_{n-1}\in I\Big\}.$$
\begin{definition}
\emph{Let $X$ be a compact space and $S$ be a subset of $\ov{X}$. For any open cover $\c{U}$ of $X$ the
\emph{topological entropy of $S$ relative to $\c{U}$} is defined to be
\begin{equation*}\label{2003062220}
\r{H}^\r{t}(S,\c{U}):=\limsup_{n\to\infty}\frac{1}{n}\log\r{N}_{S/\ov{X}}(\ov{\c{U}}^n).
\end{equation*}
The \emph{topological entropy of $S$} (as a subset of $\ov{X}$) is defined by
\begin{equation*}\label{2003311400}
\r{H}^\r{t}_X(S)=\r{H}^\r{t}(S):=\sup_\c{U}\r{H}^\r{t}(S,\c{U})
\end{equation*}
where the supremum is taken over all open covers $\c{U}$ of $X$.}
\end{definition}
For any (not necessarily continuous) mapping $T:X\to X$ let $\ov{\r{Gr}}T$ denote the subset of $\ov{X}$ given by
$\{(T^nx)_{n\geq0}:x\in X\}$. Similarly if $T$ is a set-valued mapping on $X$ let
$$\ov{\r{Gr}}T:=\Big\{(x_n)_{n\geq0}:x_0\in X, x_{n+1}\in Tx_n\Big\}\subseteq\ov{X}.$$
\begin{proposition}\label{200203062239}
\emph{Let $X$ be a compact space and $T:X\to X$ be a continuous mapping. Then for any open cover $\c{U}$ of $X$
we have $\r{h}(T,\c{U})=\r{H}^\r{t}(\ov{\r{Gr}}T,\c{U})$. Hence
$$\r{h}(T)=\r{H}^\r{t}(\ov{\r{Gr}}T).$$}
\end{proposition}
\begin{proof}
It follows from the easily checked identity $\r{N}(\vee_{i=0}^{n-1}T^{-i}\c{U})=\r{N}_{\ov{\r{Gr}}T/\ov{X}}(\ov{\c{U}}^n)$.
\end{proof}
Proposition \ref{200203062239} justifies that $\r{H}^\r{t}_X(\ov{\r{Gr}}T)$ is a natural generalization of the concept of topological
entropy for any non-continuous (set-valued) mapping $T:X\to X$.
\section{Basic Properties of $\r{H}^\r{t}_X$ and some Examples}\label{2005240803}
In this section, we consider some elementary properties and examples of our generalized topological entropy.
Throughout this section, $X$ and $Y$ denote compact Hausdorff spaces.
\begin{theorem}
\emph{For any finite subset $S$ of $\ov{X}$ we have $\r{H}^\r{t}(S)=0$.}
\end{theorem}
\begin{proof}
For any open cover $\c{U}$ of $X$ and every $n\geq0$ we have $\r{N}_{S/\ov{X}}(\ov{\c{U}}^n)\leq|S|$ and hence $\r{H}^\r{t}(S,\c{U})=0$.
Thus $\r{H}^\r{t}(S)=0$.
\end{proof}
\begin{lemma}\label{2003312345}
\emph{Let $S\subseteq\ov{X}$. For open covers $\c{U}\preccurlyeq\c{V}$ of $X$ we have $\r{H}^\r{t}(S,\c{U})\leq\r{H}^\r{t}(S,\c{V})$.}
\end{lemma}
\begin{proof}
It follows immediately from the obvious inequality $\r{N}_{S/\ov{X}}(\ov{\c{U}}^n)\leq\r{N}_{S/\ov{X}}(\ov{\c{V}}^n)$.
\end{proof}
\begin{theorem}\label{2003312302}
\emph{For $S\subseteq T\subseteq\ov{X}$ we have
$$\r{H}^\r{t}(S)\leq\r{H}^\r{t}(T).$$
If $S$ is finite then $\r{H}^\r{t}(T\setminus S)=\r{H}^\r{t}(T)$.}
\end{theorem}
\begin{proof}
Let $\c{U}$ be an open cover for $X$. We have $\r{N}_{S/\ov{X}}(\ov{\c{U}}^n)\leq\r{N}_{T/\ov{X}}(\ov{\c{U}}^n)$. This implies that
$\r{H}^\r{t}(S,\c{U})\leq\r{H}^\r{t}(T,\c{U})$. Thus $\r{H}^\r{t}(S)\leq\r{H}^\r{t}(T)$. Suppose now that $S$ is finite. We have
$$\r{N}_{T/\ov{X}}(\ov{\c{U}}^n)\leq|S|+\r{N}_{(T\setminus S)/\ov{X}}(\ov{\c{U}}^n)\hspace{10mm}(n\geq0).$$
Thus $\r{H}^\r{t}(T,\c{U})\leq\r{H}^\r{t}(T\setminus S,\c{U})$. The proof is complete.
\end{proof}
For a set $S$ in a space we denote by $\r{cls}(S)$ its closure.
\begin{lemma}\label{2006020745}
\emph{Let $\c{U},\c{V}$ be open coverings of $X$ such that for every $V\in\c{V}$ there exists $U\in\c{U}$ with
$\r{cls}(V)\subseteq U$. Then for any $S\subseteq X$ we have
$$\r{N}_{\r{cls}(S)/X}(\c{U})\leq\r{N}_{S/X}(\c{V}).$$}
\end{lemma}
\begin{proof}
Suppose that $V_1,\ldots,V_k$ in $\c{V}$ be such that $S\subseteq\cup_{i=1}^kV_i$. Then $\r{cls}(S)\subseteq\cup_{i=1}^k\r{cls}(V_i)$.
Thus if $U_i$ in $\c{U}$ be such that $\r{cls}(V_i)\subseteq U_i$ then $\r{cls}(S)\subseteq\cup_{i=1}^kU_i$. This implies that
$\r{N}_{\r{cls}(S)/X}(\c{U})\leq k$. The proof is complete.
\end{proof}
\begin{theorem}
\emph{For every $S\subseteq\ov{X}$ we have
$$\r{H}^\r{t}\big(\r{cls}(S)\big)=\r{H}^\r{t}\big(S\big).$$}
\end{theorem}
\begin{proof}
Let $\c{U}$ be an arbitrary open cover for $X$. Since $X$ is compact and Hausdorff there exists an open cover $\c{V}$ of $X$
such that for every $V\in\c{V}$ there is $U\in\c{U}$ with $\r{cls}(V)\subseteq U$. It follows that for every $n\geq1$ the closure of every member
of $\ov{\c{V}}^n$ is contained in a member of $\ov{\c{U}}^n$. Thus by Lemma \ref{2006020745} we have
$$\r{N}_{\r{cls}(S)/\ov{X}}(\ov{\c{U}}^n)\leq\r{N}_{S/\ov{X}}(\ov{\c{V}}^n).$$
This implies that $\r{H}^\r{t}\big(\r{cls}(S),\c{U}\big)\leq\r{H}^\r{t}\big(S,\c{V}\big)$ and hence $\r{H}^\r{t}\big(\r{cls}(S),\c{U}\big)\leq\r{H}^\r{t}\big(S\big)$. Thus we have
$\r{H}^\r{t}\big(\r{cls}(S)\big)\leq\r{H}^\r{t}\big(S\big)$. The revers direction follows from Theorem \ref{2003312302}.
\end{proof}
\begin{theorem}\label{2004012335}
\emph{Let $X$ be a discrete finite space and $z\in X$. Let $\tilde{X}_z$ denote the set of those sequences $(x_n)_n$ in $\ov{X}$ such that
there exists $n_0$ with $x_n=z$ for every $n\geq n_0$. Then
$$\r{H}^\r{t}(\tilde{X}_z)=\r{H}^\r{t}(\ov{X})=\log|X|.$$
In particular for every $S\subseteq\ov{X}$, $\r{H}^\r{t}(S)\leq\log|X|$.}
\end{theorem}
\begin{proof}
Let $\c{U}$ denote the maximal open cover of $X$ with respect to $\preccurlyeq$ i.e. $\c{U}=\{\{x\}\}_{x\in X}$.
It follows from Lemma \ref{2003312345} that $\r{H}^\r{t}(\ov{X})=\r{H}^\r{t}(\ov{X},\c{U})=\lim_{n\to\infty}\frac{1}{n}\log|X|^n=\log|X|$.
The same proof shows that $\r{H}^\r{t}(\tilde{X}_z)=\log|X|$. The last part follows from Theorem \ref{2003312302}.
\end{proof}
\begin{theorem}\label{2005240830}
\emph{Let $X$ be infinite and for $z\in X$ let $\tilde{X}_z\subset\ov{X}$ be as in Theorem \ref{2004012335}.
Then $$\r{H}^\r{t}(\ov{X})=\r{H}^\r{t}(\tilde{X}_z)=\infty.$$}
\end{theorem}
\begin{proof}
Let $k\geq2$ be arbitrary and let $x_1,\ldots,x_k$ be distinct elements of $X$. There is an open cover $\c{U}$ of $X$ such that every member
of $\c{U}$ contains at most only one of the $x_1,\ldots,x_k$. We have
$$\r{H}^\r{t}(\tilde{X}_z)\geq\r{H}^\r{t}(\tilde{X}_z,\c{U})\geq\lim_{n\to\infty}\frac{1}{n}\log k^n=\log k.$$
Thus $\r{H}^\r{t}(\tilde{X}_z)=\infty$.
\end{proof}
\begin{theorem}\label{2004022145}
\emph{Let $Y$ be a (closed) subspace of $X$ and $S\subseteq\ov{Y}$. Then
$$\r{H}^\r{t}_Y(S)=\r{H}^\r{t}_X(S).$$}
\end{theorem}
\begin{proof}
Let $\c{U}=\{U_i\}$ be an open cover for $X$. Then $\c{U}\cap Y:=\{U_i\cap Y\}$
is an open cover for $Y$ and we have $\r{N}_{S/\ov{X}}(\ov{\c{U}}^n)=\r{N}_{S/\ov{Y}}(\ov{\c{U}\cap Y}^n)$.
This implies that $\r{H}^\r{t}(S,\c{U})=\r{H}^\r{t}(S,\c{U}\cap Y)$ and hence $\r{H}^\r{t}_X(S)\leq\r{H}^\r{t}_Y(S)$.
Let $\c{V}=\{V_i\}$ be an open cover for $Y$.
Let $U_i$ denote an open subset of $X$ with $V_i=U_i\cap Y$. Then $\c{U}:=\{U_i\}\cup\{X\setminus Y\}$ is an open cover for $X$ and
we have $\r{N}_{S/\ov{X}}(\ov{\c{U}}^n)=\r{N}_{S/\ov{Y}}(\ov{\c{V}}^n)$.
This implies that $\r{H}^\r{t}(S,\c{U})=\r{H}^\r{t}(S,\c{V})$ and hence $\r{H}^\r{t}_X(S)\geq\r{H}^\r{t}_Y(S)$.
The proof is complete.
\end{proof}
Here is another proof of Theorem \ref{2005240830}: For any $k$ let $Y=\{x_1,\ldots,x_{k-1},z\}$ be a subset of
$X$ with $k$ elements. By Theorem \ref{2004012335}, $\r{H}^\r{t}_Y(\tilde{Y}_z)=\log k$, by Theorem \ref{2004022145},
$\r{H}^\r{t}_X(\tilde{Y}_z)=\r{H}^\r{t}_Y(\tilde{Y}_z)$, and by Theorem \ref{2003312302},
$\r{H}^\r{t}_X(\tilde{X}_z)\geq\r{H}^\r{t}_X(\tilde{Y}_z)$. Thus $\r{H}^\r{t}_X(\tilde{X}_z)=\infty$.

We now want to show that every positive real number less than $\log|X|$ appears as the entropy of a subset of $\ov{X}$.
For this aim we need the following two lemmas.
\begin{lemma}\label{2005240840}
\emph{Suppose that $X$ is finite and $z\in X$. Let $r\leq1$ be a positive real number. Suppose that there exist two sequences $(p_k)_k,(q_k)_k$
of natural numbers satisfying the following conditions:
\begin{enumerate}
\item[(i)] $p_k\leq q_k$,  $p_{k}<p_{k+1}$, and $q_{k}<q_{k+1}$.
\item[(ii)] $(q_{k+1}-q_k)\geq(p_{k+1}-p_k)$ and $\frac{p_k}{q_k}\leq\frac{p_{k+1}}{q_{k+1}}$.
\item[(iii)] $\sup_k\frac{p_k}{q_k}=r$.
\end{enumerate}
Then there exists a subset $S_r\subseteq\ov{X}$ such that $\r{H}^\r{t}(S_r)=r\log|X|$.}
\end{lemma}
\begin{proof}
Let $S_r\subseteq\ov{X}$ be defined by
$$S_r:=\underbrace{\{z\}\times\cdots\{z\}\times\overbrace{X\times\cdots\times X}^{p_1}}_{q_1}\times
\underbrace{\{z\}\times\cdots\times\{z\}\times\overbrace{X\times\cdots\times X}^{p_2-p_1}}_{q_2-q_1}\times\cdots$$
Let $\c{U}$ be as in the proof of Theorem \ref{2004012335}. Then $\r{N}_{S_r/\ov{X}}(\ov{\c{U}}^{q_k})=|X|^{p_k}$. This shows that
$$\r{H}^\r{t}(S_r)=\r{H}^\r{t}(S_r,\c{U})=\sup_k\frac{p_k}{q_k}\log|X|=r\log|X|.$$
\end{proof}
\begin{lemma}\label{2004121440}
\emph{Suppose that $r$ is a positive real number with $r\leq1$. Then there exist sequences $(p_k)_k,(q_k)_k$ satisfying the conditions
(i)-(iii) of Lemma \ref{2005240840}.}
\end{lemma}
\begin{proof}
For every two natural numbers $a,b$ and any rational number $s$ satisfying $\frac{a}{b}\leq s\leq1$
it is easily verified that there exist natural numbers $a',b'$ such that
\begin{equation}\label{2004121436}
a<a',\hspace{2mm} b<b',\hspace{2mm} (b'-b)\geq(a'-a),\hspace{2mm}\frac{a'}{b'}=s.
\end{equation}
Let $(r_k)_k$ be a sequence of positive rational numbers such that $r_{k}\leq r_{k+1}$ and $\sup_kr_k=r$.
We define inductively the desired sequences $(p_k)_k,(q_k)_k$ as follows.
Let $p_1,q_1$ be natural numbers with $\frac{p_1}{q_1}=r_1$. Suppose that we have chosen natural numbers $p_1,\ldots,p_\ell,q_1,\ldots,q_\ell$
such that the inequalities in (i) and (ii) of Lemma \ref{2005240840} are satisfied for $k=1,\ldots\ell-1$.
Then we let $p_{\ell+1}:=a',q_{\ell+1}:=b'$ where $a',b'$ satisfy (\ref{2004121436}) for $a=p_\ell,b=q_\ell$, and $s=r_{\ell+1}$.
The proof is complete.
\end{proof}
\begin{theorem}
\emph{For any compact space $X$, every number $u$ with $0\leq u\leq\log|X|$ appears as the topological entropy of a subset of $\ov{X}$.}
\end{theorem}
\begin{proof}
First suppose that $X$ is finite. Thus $\log|X|<\infty$. Let $r:=u/\log|X|$. By Lemmas \ref{2005240840} and \ref{2004121440}
there exists a set $S_r\subseteq\ov{X}$ such that $\r{H}^\r{t}(S_r)=r\log|X|=u$. For infinite $X$
by Theorem \ref{2005240830} it is enough to suppose that $u<\infty$. Then the desired result is concluded from the result
for a finite subspace $Y$ of $X$ with $u\leq\log|Y|$ and Theorem \ref{2004022145}.
\end{proof}
\begin{theorem}\label{2004022204}
\emph{Let $S\subseteq\ov{X},T\subseteq\ov{Y}$. Let $X\dot{\cup}Y$ denote the disjoint union space of $X$ and $Y$.
Consider $S\dot{\cup}T$ as a subset of $\ov{X\dot{\cup}Y}$ in the obvious way. Then
$$\r{H}^\r{t}_{X\dot{\cup}Y}(S\dot{\cup}T)=\max\{\r{H}^\r{t}_X(S),\r{H}^\r{t}_Y(T)\}.$$}
\end{theorem}
\begin{proof}
Let $\c{U},\c{V}$ be open covers respectively for $X,Y$. Then $\c{U}\cup\c{V}$ is an open cover for $X\dot{\cup}Y$.
Note also that every open cover of $X\dot{\cup}Y$ is of this form. We have
$$\r{N}_{S\dot{\cup}T/\ov{X\dot{\cup}Y}}(\ov{\c{U}\cup\c{V}}^n)=\r{N}_{S/\ov{X}}(\ov{\c{U}}^n)+\r{N}_{T/\ov{Y}}(\ov{\c{V}}^n).$$
Thus we have
$$\log[\r{N}_{S\dot{\cup}T/\ov{X\dot{\cup}Y}}(\ov{\c{U}\cup\c{V}}^n)]
\leq\max\Big\{\log[\r{N}_{S/\ov{X}}(\ov{\c{U}}^n)]+\log2,\log[\r{N}_{T/\ov{Y}}(\ov{\c{V}}^n)]+\log2\Big\}.$$
This implies that $\r{H}^\r{t}_{X\dot{\cup}Y}(S\dot{\cup}T)\leq\max\{\r{H}^\r{t}_X(S),\r{H}^\r{t}_Y(T)\}$. The reverse direction follows from
Theorems \ref{2003312302} and \ref{2004022145}.
\end{proof}
We identify $\ov{X}\times\ov{Y}$ with $\ov{X\times Y}$ in the obvious way via the mapping
$$\big((x_n)_n,(y_n)_n\big)\mapsto\big(x_n,y_n\big)_n.$$
\begin{theorem}\label{2003312200}
\emph{Let $S\subseteq\ov{X},T\subseteq\ov{Y}$. Then
$$\r{H}^\r{t}_{X\times Y}(S\times T)\leq\r{H}^\r{t}_X(S)+\r{H}^\r{t}_Y(T).$$}
\end{theorem}
\begin{proof}
Let $\c{U},\c{V}$ be open covers respectively for $X,Y$ and let $\c{U}\times\c{V}$ denote the associated product cover.
We have $\r{N}_{S\times T/\ov{X\times Y}}(\ov{\c{U}\times\c{V}}^n)=\r{N}_{S/\ov{X}}(\ov{\c{U}}^n)\r{N}_{T/\ov{Y}}(\ov{\c{V}}^n)$.
This implies that
\begin{equation}\label{2003312349}
\r{H}^\r{t}(S\times T,\c{U}\times\c{V})\leq\r{H}^\r{t}(S,\c{U})+\r{H}^\r{t}(T,\c{V}).
\end{equation}
On the other hand if $\c{W}$ is an open cover for $X\times Y$ then there exist open covers $\c{U},\c{V}$ for $X,Y$ such that
$\c{W}\preccurlyeq\c{U}\times\c{V}$ and hence by Lemma \ref{2003312345}, $\r{H}^\r{t}(S\times T,\c{W})\leq\r{H}^\r{t}(S\times T,\c{U}\times\c{V})$.
Thus the inequality (\ref{2003312349}) implies the desired result.
\end{proof}
The generalized topological entropy is \emph{shift invariant}:
\begin{theorem}
\emph{For $S\subseteq\ov{X}$ and any $k\geq1$ let $\f{S}^kS\subseteq\ov{X}$ be obtained by $k$-times action of the shift map on $S$:
$$\f{S}^kS:=\Big\{(x_n)_{n\geq0}\big|(y_0,\ldots,y_{k-1},x_0,x_1,\ldots)\in S\hspace{2mm}
\text{for some}\hspace{2mm}y_0,\ldots,y_{k-1}\in X\Big\}.$$
Then
$$\r{H}^\r{t}_X(\f{S}^kS)=\r{H}^\r{t}_X(S).$$}
\end{theorem}
\begin{proof}
Let $\c{U}$ be an arbitrary open cover for $X$. Without loss of generality we may suppose that $|\c{U}|<\infty$. Then
$$\r{N}_{\f{S}^kS/\ov{X}}(\ov{\c{U}}^n)\leq\r{N}_{S/\ov{X}}(\ov{\c{U}}^{k+n})\leq|\c{U}|^k\r{N}_{\f{S}^kS/\ov{X}}(\ov{\c{U}}^n).$$
This implies that $\r{H}^\r{t}(\f{S}^kS,\c{U})=\r{H}^\r{t}(S,\c{U})$. Hence
$\r{H}^\r{t}_X(\f{S}^kS)=\r{H}^\r{t}_X(S)$.
\end{proof}
\begin{theorem}
\emph{Let $S\subseteq\ov{X}$. For $k\geq2$ we identify $\ov{X^k}$ with $\ov{X}$ in a obvious way via the mapping
$$\big((x_n^1,\ldots,x_n^k)\big)_{n\geq0}\mapsto\big(x_0^1,\ldots,x_0^k,x_1^1,\ldots,x^k_1,x^1_2,\ldots\big).$$
Then
$$\r{H}^\r{t}_{X^k}(S)=k\r{H}^\r{t}_X(S).$$}
\end{theorem}
\begin{proof}
Let $\c{U}$ be an open cover for $X$ and consider the associated product cover $\c{U}^k$ of $X^k$.
For $n\geq1$ let $[\frac{n}{k}]$ denote the greatest integer less than or equal to $\frac{n}{k}$. It is easily verified that for every $n\geq k$,
$$\r{N}_{S/\ov{X^k}}(\ov{\c{U}^k}^{[\frac{n}{k}]})\leq\r{N}_{S/\ov{X}}(\ov{\c{U}}^{n})\leq\r{N}_{S/\ov{X^k}}(\ov{\c{U}^k}^{[\frac{n}{k}]+1}).$$
Thus we have
$$\frac{k}{n}\log\r{N}_{S/\ov{X}}(\ov{\c{U}}^{n})\leq\frac{1}{[\frac{n}{k}]}\log\r{N}_{S/\ov{X^k}}(\ov{\c{U}^k}^{[\frac{n}{k}]+1}).$$
This implies that
$$k\r{H}^\r{t}(S,\c{U})=\limsup_{n\to\infty}\frac{k}{n}\log\r{N}_{S/\ov{X}}(\ov{\c{U}}^{n})\leq
\limsup_{\ell\to\infty}\frac{1}{\ell}\log\r{N}_{S/\ov{X^k}}(\ov{\c{U}^k}^{\ell+1})=\r{H}^\r{t}(S,\c{U}^k).$$
Similarly it is proved that $\r{H}^\r{t}(S,\c{U}^k)\leq k\r{H}^\r{t}(S,\c{U})$. Thus we have
\begin{equation}\label{2003312201}
\r{H}^\r{t}(S,\c{U}^k)=k\r{H}^\r{t}(S,\c{U}).
\end{equation}
If $\c{V}$ is an open cover for $X^k$ then there is an open cover $\c{U}$ for $X$ such that $\c{V}\preccurlyeq\c{U}^k$ and
hence by Lemma \ref{2003312345}, $\r{H}^\r{t}(S,\c{V})\leq\r{H}^\r{t}(S,\c{U}^k)$.
Thus (\ref{2003312201}) implies $\r{H}^\r{t}_{X^k}(S)=k\r{H}^\r{t}_X(S)$.
\end{proof}
\begin{theorem}\label{2004121500}
\emph{Let $S\subseteq\ov{X}$. For $k\geq2$ let $\f{R}^kS\subseteq\ov{X}$ denote the \emph{$k$-restriction} of $S$:
$$\f{R}^kS:=\Big\{(x_{kn})_{n\geq0}|(x_n)_{n\geq0}\in S\Big\}.$$
Then we have
$$\r{H}^\r{t}(\f{R}^kS)\leq k\r{H}^\r{t}(S).$$}
\end{theorem}
\begin{proof}
Let $\c{U}$ be an open cover for $X$. We have
$\r{N}_{\f{R}^kS/\ov{X}}(\ov{\c{U}}^n)\leq\r{N}_{S/\ov{X}}(\ov{\c{U}}^{kn})$. This implies that
$\r{H}^\r{t}(\f{R}^kS,\c{U})\leq k\r{H}^\r{t}(S,\c{U})$. Thus $\r{H}^\r{t}(\f{R}^kS)\leq k\r{H}^\r{t}(S)$.
\end{proof}
\begin{theorem}
\emph{Let $S\subseteq\ov{X}$. For $k\geq2$ let $\f{D}^kS\subseteq\ov{X}$ denote the \emph{$k$-dilation} of $S$:
$$\f{D}^kS:=\Big\{(\overbrace{x_0,\ldots,x_0}^k,\overbrace{x_1,\ldots,x_1}^k,\ldots): (x_n)_{n\geq0}\in S\Big\}.$$
Then
$$\r{H}^\r{t}(\f{D}^kS)=\frac{1}{k}\r{H}^\r{t}(S).$$}
\end{theorem}
\begin{proof}
Let $\c{U}$ be an open cover for $X$. For every $n\geq1$ and $i=0,\ldots,k-1$ we have
$$\r{N}_{\f{D}^kS/\ov{X}}(\ov{\c{U}}^{nk-i})=\r{N}_{S/\ov{X}}(\ov{\c{U}}^{n}).$$
Thus for $m=nk-i$ we have
$$(k-\frac{i}{n})\frac{1}{m}\log\r{N}_{\f{D}^kS/\ov{X}}(\ov{\c{U}}^{m})=\frac{1}{n}\r{N}_{S/\ov{X}}(\ov{\c{U}}^{n}).$$
This implies that $\r{H}^\r{t}(S,\c{U})=k\r{H}^\r{t}(\f{D}^kS,\c{U})$. Hence $\r{H}^\r{t}(S)=k\r{H}^\r{t}(\f{D}^kS)$.
\end{proof}
By the following simple result we may find upper bounds for the entropy of a (not small) class of noncontinuous mappings.
(See Example \ref{2005250707} below.)
\begin{theorem}\label{2004121700}
\emph{Let $T:X\to X$ be a (not necessarily continuous) mapping and $K$ be a nonempty proper subset of $X$ such that $K,X\setminus K$ are
invariant under $T$ i.e. $T(K)\subseteq K,T(X\setminus K)\subseteq X\setminus K$. For a fixed $z\in K$ let the mapping $T_{K,z}:X\to X$
be defined by $T_{K,z}(x)=T(x)$ if $x\in X\setminus K$ and otherwise $T_{K,z}(x)=z$. Then
$$\r{H}^\r{t}(\ov{\r{Gr}}T_{K,z})\leq\r{H}^\r{t}(\ov{\r{Gr}}T).$$}
\end{theorem}
\begin{proof}
Straightforward.
\end{proof}
\begin{theorem}\label{2004021402}
\emph{Let $\phi:X\to Y$ be a continuous map and $S\subseteq\ov{X}$. Let $\phi S\subseteq\ov{Y}$ denote the subset $\{(\phi x_n)_n:(x_n)_n\in S\}$.
Then for any open cover $\c{V}$ of $Y$, $\r{H}^\r{t}(S,\phi^{-1}\c{V})=\r{H}^\r{t}(\phi S,\c{V})$. Thus we have
$$\r{H}^\r{t}_Y(\phi S)\leq\r{H}^\r{t}_X(S).$$
If $\phi$ is a homeomorphism then
$$\r{H}^\r{t}_Y(\phi S)=\r{H}^\r{t}_X(S).$$}
\end{theorem}
\begin{proof}
Straightforward.
\end{proof}
Note that the inequality in Theorem \ref{2004021402} can not be equality in general: If $X$ is infinite and $Y$ is a singleton space then
$\r{H}^\r{t}(\ov{X})=\infty$ and $\r{H}^\r{t}(\phi\ov{X})=0$.
\begin{theorem}
\emph{Let $S,S'\subseteq\ov{X}$. Then
$$\r{H}^\r{t}(S\cup S')=\max\{\r{H}^\r{t}(S),\r{H}^\r{t}(S')\}.$$}
\end{theorem}
\begin{proof}
By Theorem \ref{2004022204}, $\r{H}^\r{t}_{X\dot{\cup}X}(S\dot{\cup}S')=\max\{\r{H}^\r{t}(S),\r{H}^\r{t}(S')\}$. The canonical mapping $\phi:X\dot{\cup}X\to X$ is continuous and $\phi(S\dot{\cup}S')=S\cup S'$. Thus it follows from Theorem \ref{2004021402} that
$\r{H}^\r{t}(S\cup S')\leq\max\{\r{H}^\r{t}(S),\r{H}^\r{t}(S')\}$.
The reverse follows from Theorem \ref{2003312302}.
\end{proof}
For any topological space $Z$ let $\r{CL}(Z)$ denote the set of all nonempty closed subsets of $Z$. Recall that the \emph{upper Vietoris topology}
on $\r{CL}(Z)$ is the smallest topology containing all the subsets of $\r{CL}(Z)$ of the form $\{K\in\r{CL}(Z):K\subseteq U\}$
where $U\subseteq Z$ is open.
\begin{theorem}
\emph{Let $(S_\lambda)_\lambda$ be a net in $\r{CL}(\ov{X})$ converging to a closed subset $S$ of $\ov{X}$ in upper Vietoris topology.
Let $\c{U}$ be an open cover for $X$. Suppose that for every $\lambda$ the $\limsup_n$ in the definition of $\r{H}^\r{t}(S,\c{U})$
(given by (\ref{2003062220})) is equal to $\inf_n$ (e.g. for every $\lambda$ there is a continuous mapping $T_\lambda:X\to X$ such that
$S_\lambda=\ov{\r{Gr}}T$). Then
$$\limsup_\lambda\r{H}^\r{t}(S_\lambda,\c{U})\leq\r{H}^\r{t}(S,\c{U}).$$}
\end{theorem}
\begin{proof}
Without loss of generality suppose that $\c{U}$ is finite and let $\c{U}=\{U_1,\ldots,U_\ell\}$.
For a fixed $n\geq1$ let $k=\r{N}_{S/\ov{X}}(\ov{\c{U}}^n)$. For $i=1,\ldots,k$ and $j=0,\ldots,n-1$
there is an index $q^i_j$ with $1\leq q^i_j\leq\ell$ such that
$$S\subseteq V:=\bigcup_{i=1}^kU_{q^i_0}\times\cdots\times U_{q^i_{n-1}}\times X\times X\times\cdots$$
Since $V$ is open there is $\lambda_n$ such that for every $\lambda\geq\lambda_n$, $S_\lambda\subseteq V$ and hence
$\r{N}_{S_\lambda/\ov{X}}(\ov{\c{U}}^n)\leq k$. This implies that
$$\r{H}^\r{t}(S_\lambda,\c{U})\leq\frac{1}{n}\log\r{N}_{S_\lambda/\ov{X}}(\ov{\c{U}}^n)\leq\frac{1}{n}\log k=
\frac{1}{n}\log\r{N}_{S/\ov{X}}(\ov{\c{U}}^n)\hspace{5mm}(\lambda\geq\lambda_n).$$
The desired result follows from the above inequality.
\end{proof}
Another easy result related to Vietoris topology is as follows:
\begin{theorem}
\emph{Let $(Z_n)_{n\geq0}$ be a sequence in $\r{CL}(X)$ converging to a singleton $\{z\}$ in upper Vietoris topology. Then for
$S:=\prod_{n=0}^\infty Z_n$ we have $\r{H}^\r{t}_X(S)=0$.}
\end{theorem}
\begin{proof}
Straightforward.
\end{proof}
The following is a simple generalization of \cite[Theorem 7.6]{Walters1}. (If $X$ is endowed with a compatible metric then for an
open cover $\c{U}$ of $X$, $\r{diam}(\c{U})$ denotes the supremum of diameters of members of $\c{U}$.)
\begin{theorem}\label{2004270830}
\emph{Let $(X,d)$ be a compact metric space and $S\subseteq\ov{X}$. Suppose that $(\c{U}_i)_i$ is a sequence of open covers of $X$ such that
$\lim_{i\to\infty}\r{diam}(\c{U}_i)=0$. Then
$$\r{H}^\r{t}(S)=\lim_{i\to\infty}\r{H}^\r{t}(S,\c{U}_i).$$}
\end{theorem}
\begin{proof}
Suppose that $\r{H}^\r{t}(S)<\infty$. Let $\epsilon>0$ be arbitrary and let $\c{V}$ be an open cover for $X$ such that
$\r{H}^\r{t}(S)<\r{H}^\r{t}(S,\c{V})+\epsilon$. Suppose that $\delta>0$ is a Lebesgue number for $\c{V}$ and $i_0$ be such that for every
$i\geq i_0$, $\r{diam}(\c{U}_i)<\delta$. Then for $i\geq i_0$ we have $\c{V}\preccurlyeq\c{U}_i$ and hence by Lemma \ref{2003312345},
$|\r{H}^\r{t}(S)-\r{H}^\r{t}(S,\c{U}_i)|<\epsilon$. The proof in the case $\r{H}^\r{t}(S)=\infty$ is similarly.
\end{proof}
We end this section by three funny (!) examples.
\begin{example}\label{2004121600}
\emph{For an integer $k\geq2$ we may identify the unit interval $[0,1]$ with a subset of $\ov{X}$ where $X:=\{0,\ldots,k-1\}$ via the map that
associates to any number $r\in[0,1]$ the lower representation of $r$ in the basis $X$. (Thus for instance in the case that $k=2$ for $\frac{1}{2}$
we choose the representation $(01111\cdots)$ rather than $(10000\cdots)$.) Thus for every subset of $S\subseteq[0,1]$
we may consider $\r{H}^\r{t}_X(S)$ as a measure of intrinsic chaos of $S$. Similar to the proof of Theorem \ref{2004012335}
it is easily shown that the entropy of rationals in $[0,1]$ as well as irrationals is equal to $\log k$.
In the case $k=3$ the entropy of the Cantor set is $\log2$.}
\end{example}
Any labeling of points of a space $F$ by the sequences of points of another space $X$ gives rise to an entropy theory for subsets of $F$
as we saw in Example \ref{2004121600}. Here is another example of this sort from Fractal Theory:
\begin{example}\label{2004121600}
\emph{Let $F$ denote the self-similar set induced by an $N$-tuple $(f_1,\ldots,f_N)$ of contractions on a complete metric space $Z$.
Then as it is shown in \cite[Theorem 1.2.3]{Kigami1} the $N$-tuple $(f_1,\ldots,f_N)$ induces also a canonical continuous surjection $\pi:\ov{X}\to F$
where $X:=\{1,\ldots,N\}$. Thus for any subset $S$ of $F$ we may consider the value $\r{H}^\r{t}_X(\pi^{-1}S)$ as the entropy of $S$.}
\end{example}
\begin{example}\label{2005250707}
\emph{Let $T:[0,1]\to[0,1]$ be defined by $T(x)=x$ if $x$ is irrational and $T(x)=0$ if $x$ is rational. Since the topological entropy of the identity
map on $[0,1]$ is zero it follows immediately from Theorem \ref{2004121700} that the generalized topological entropy of $T$ is also zero i.e.
$\r{H}^\r{t}(\ov{\r{Gr}}T)=0$.}
\end{example}
\section{Entropy Induced by a Sequence of Semimetrics}\label{2005300733}
In this section we generalize Bowen's notion of topological entropy of continuous mappings for subsets of a space with a sequence of
semimetrics. We begin by some notations. Let $X$ be a topological space and let $\gamma$ be a \emph{compatible semimetric} on $X$
i.e. $\gamma:X\times X\to[0,\infty)$ is a function satisfying (i) $\gamma(x,x)=0$, (ii) $\gamma(x,y)=\gamma(y,x)$,
(iii) $\gamma(x,z)\leq\gamma(x,y)+\gamma(y,z)$, and (iv) the topology of $\gamma$ (induced by the family of open balls as the basis)
is contained in the original topology of $X$. For every compact subset $K$ of $X$ and any $\epsilon>0$ we let
$\r{N}^\r{d}(K,\gamma,\epsilon)$ (resp. $\r{N}^\r{id}(K,\gamma,\epsilon)$)
denote the smallest number $k$ for which there exists a set $\{x_1,\ldots,x_k\}$ of elements of $X$ (resp. of $K$) such that
for every $x\in K$, $\gamma(x,x_i)<\epsilon$ for some $i=1,\ldots,k$. Note that since $K$ is $\gamma$-compact $k$ exists.
We let $\r{N}^\r{s}(K,\gamma,\epsilon)$ denote the largest number $\ell$ for which there is a subset
$\{y_1,\ldots,y_\ell\}$ of $K$ with $\gamma(y_i,y_j)\geq\epsilon$ for $i\neq j$.
Again since $K$ is $\gamma$-compact $\ell$ exists. In the above, the superscripts $\r{d}$, $\r{id}$, and $\r{s}$ of $\r{N}$
stand respectively for \emph{density}, \emph{intrinsic density}, and \emph{separateness}.
We have the following easily verified lemma. The proof is omitted.
\begin{lemma}\label{2005260700}
\emph{Let $X,K,\gamma$ be as above. The following statements are satisfied.
\begin{enumerate}
\item[(i)] For $0<\epsilon<\epsilon'$ we have $$\r{N}^\r{d}(K,\gamma,\epsilon')\leq\r{N}^\r{d}(K,\gamma,\epsilon),\hspace{2mm}
\r{N}^\r{id}(K,\gamma,\epsilon')\leq\r{N}^\r{id}(K,\gamma,\epsilon),\hspace{2mm}
\r{N}^\r{s}(K,\gamma,\epsilon')\leq\r{N}^\r{s}(K,\gamma,\epsilon).$$
\item[(ii)] For every $\epsilon>0$, $\r{N}^\r{id}(K,\gamma,2\epsilon)\leq\r{N}^\r{d}(K,\gamma,\epsilon)\leq\r{N}^\r{id}(K,\gamma,\epsilon)$ and
$$\r{N}^\r{s}(K,\gamma,2\epsilon)\leq\r{N}^\r{id}(K,\gamma,\epsilon)\leq\r{N}^\r{s}(K,\gamma,\epsilon).$$
\item[(iii)] If $\epsilon>0$ is a Lebesgue number for an open cover $\c{U}$ of $X$ (i.e. any subset of $X$ with diameter strictly less than
$\epsilon$ is contained in a member of $\c{U}$) then
$$\r{N}_{K/X}(\c{U})\leq\r{N}^\r{id}(K,\gamma,\frac{\epsilon}{2}).$$
\item[(iv)] For an open cover $\c{U}$ of $X$ if $\r{diam}(\c{U})<\epsilon$ then $\r{N}^\r{s}(K,\gamma,\epsilon)\leq\r{N}_{K/X}(\c{U})$.
\item[(v)] Let $K'\subseteq X$ be compact with $K'\supseteq K$ and $\gamma'$ be a compatible semimetric on $X$ with $\gamma'\geq\gamma$.
Then for every $\epsilon>0$ we have $$\r{N}^\r{d}(K,\gamma,\epsilon)\leq\r{N}^\r{d}(K',\gamma',\epsilon),\hspace{2mm}
\r{N}^\r{id}(K,\gamma,\epsilon)\leq\r{N}^\r{id}(K',\gamma',\epsilon),\hspace{2mm}
\r{N}^\r{s}(K,\gamma,\epsilon)\leq\r{N}^\r{s}(K',\gamma',\epsilon).$$
\end{enumerate}}
\end{lemma}
\begin{proposition}\label{2004041424}
\emph{Let $X,K$ be as above. Let $(\gamma_n)_{n\geq1}$ be a (not necessarily increasing) sequence of compatible semimetrics on $X$. Then
\begin{equation}\label{2004032232}
\begin{split}
\lim_{\epsilon\to0}\limsup_{n\to\infty}\frac{1}{n}\log\r{N}^\r{d}(K,\gamma_n,\epsilon)&=
\lim_{\epsilon\to0}\limsup_{n\to\infty}\frac{1}{n}\log\r{N}^\r{id}(K,\gamma_n,\epsilon)\\
&=\lim_{\epsilon\to0}\limsup_{n\to\infty}\frac{1}{n}\log\r{N}^\r{s}(K,\gamma_n,\epsilon).
\end{split}
\end{equation}}
\end{proposition}
\begin{proof}
By Lemma \ref{2005260700}(i) if $\epsilon\leq\epsilon'$ then
$$\limsup_{n\to\infty}\frac{1}{n}\log\r{N}^\r{d}(K,\gamma_n,\epsilon')\leq\limsup_{n\to\infty}\frac{1}{n}\log\r{N}^\r{d}(K,\gamma_n,\epsilon).$$
This shows that the first $\lim_{\epsilon}$ in (\ref{2004032232}) exists an equals to $\sup_{\epsilon}$.
Similarly it is proved that the other two $\lim$s in (\ref{2004032232}) exist and are $\sup$s. By Lemma \ref{2005260700}(ii) we have
\begin{equation}\label{2005260633}
\frac{1}{n}\log\r{N}^\r{id}(K,\gamma_n,2\epsilon)\leq\frac{1}{n}\log\r{N}^\r{d}(K,\gamma_n,\epsilon)\leq
\frac{1}{n}\log\r{N}^\r{id}(K,\gamma_n,\epsilon).
\end{equation}
\begin{equation}\label{2004211430}
\frac{1}{n}\log\r{N}^\r{s}(K,\gamma_n,2\epsilon)\leq\frac{1}{n}\log\r{N}^\r{id}(K,\gamma_n,\epsilon)\leq
\frac{1}{n}\log\r{N}^\r{s}(K,\gamma_n,\epsilon).
\end{equation}
Now (\ref{2004032232}) follows from (\ref{2005260633}) and (\ref{2004211430}).
\end{proof}
\begin{definition}\label{2005300731}
\emph{Let $X$ be a topological space and $\Gamma=(\gamma_n)_{n\geq1}$ be an increasing sequence of compatible semimetrics on $X$.
For every compact $K\subseteq X$ we denote the common value in (\ref{2004032232}) by $\r{H}^\r{B}_{\Gamma}(K)$.
For any subset $S$ of $X$ the \emph{Bowen entropy of $S$ with respect to $\Gamma$} is defined to be
$$\r{H}^\r{B}_\Gamma(S):=\sup\r{H}^\r{B}_{\Gamma}(K)$$
where the supremum is taken over all compact subsets $K$ of $X$ contained in $S$.}
\end{definition}
If $K,K'$ are compact with $K\subseteq K'$ then $\r{H}^\r{B}_{\Gamma}(K)\leq\r{H}^\r{B}_{\Gamma}(K')$.
(See Theorem \ref{2005310808}(iii) below.) Thus the notion of $\r{H}^\r{B}_\Gamma$ is well-defined.
Also it is easily seen from (\ref{2004032232}) that the value $\r{H}^\r{B}_{\Gamma}(K)$ is independent from $X$. This means that
if $K\subseteq X'\subseteq X$ then the Bowen entropies of $K$ as subsets of $X$ and of $X'$ coincide. Thus to find $\r{H}^\r{B}_\Gamma(S)$
we may ignore $X$ and restrict $\gamma_n$ to $S$.

In the reminder of this section we consider some basic properties of the notion $\r{H}^\r{B}_\Gamma$.
\begin{theorem}\label{2005310808}
\emph{Let $\Gamma=(\gamma_n)_n,\Gamma'=(\gamma_n')_n$ be increasing sequences of compatible semimetrics on $X$ and $S,S'\subseteq X$.
The following statements hold.
\begin{enumerate}
\item[(i)] If $\Gamma$ is a constant sequence (i.e. $\gamma_n=\gamma$ for every $n$) then $\r{H}^\r{B}_{\Gamma}(S)=0$.
\item[(ii)] For a constant $c>0$ let $c\Gamma$ denote the sequence $(c\gamma_n)_n$. Then $\r{H}^\r{B}_{c\Gamma}(S)=\r{H}^\r{B}_{\Gamma}(S)$.
\item[(iii)] If $S\subseteq S'$ and $\gamma_n\leq\gamma'_n$ for every $n$ then $\r{H}^\r{B}_{\Gamma}(S)\leq\r{H}^\r{B}_{\Gamma'}(S')$.
\item[(iv)] If there is a constant $c>0$ such that $\gamma_n\leq c\gamma'_n$ for every $n$ then
$\r{H}^\r{B}_{\Gamma}(S)\leq\r{H}^\r{B}_{\Gamma'}(S)$.
\item[(v)] If $\Gamma$ and $\Gamma'$ are \emph{equi strongly-uniform-equivalent} i.e.
there exist constants $c,c'>0$ such that $c'\gamma'_n\leq\gamma_n\leq c\gamma'_n$ for every $n$, then
$\r{H}^\r{B}_{\Gamma}(S)=\r{H}^\r{B}_{\Gamma'}(S)$.
\item[(vi)] If $\sup_n\gamma_n$ is a compatible semimetric on $X$ then $\r{H}^\r{B}_{\Gamma}(S)=0$.
\end{enumerate}}
\end{theorem}
\begin{proof}
(i) is trivial. (ii) follows from $\r{N}^\r{id}(K,c\gamma_n,c\epsilon)=\r{N}^\r{id}(K,\gamma_n,\epsilon)$.
(iii) follows from Lemma \ref{2005260700}(v). (iv) follows from (ii) and (iii). (v) follows from (iv). (vi) follows from (i) and (iii).
\end{proof}
\begin{theorem}\label{2006010740}
\emph{Let $\Gamma,\Gamma',S$ be as in Theorem \ref{2005310808}.
\begin{enumerate}
\item[(i)] Suppose that for every $\epsilon>0$ there exists $\delta>0$ such that for every $n$  and every $x,x'\in X$ if
$\gamma'_n(x,x')<\delta$ then $\gamma_n(x,x')<\epsilon$. Then $\r{H}^\r{B}_{\Gamma}(S)\leq\r{H}^\r{B}_{\Gamma'}(S)$.
\item[(ii)] Suppose that $\Gamma$ and $\Gamma'$ are \emph{equi uniformly-equivalent}
i.e. for every $\epsilon>0$ there exists $\delta>0$ such that for every $n$ and every $x,x'\in X$ if $\gamma'_n(x,x')<\delta$
then $\gamma_n(x,x')<\epsilon$, and if $\gamma_n(x,x')<\delta$ then $\gamma'_n(x,x')<\epsilon$. Then
$$\r{H}^\r{B}_{\Gamma}(S)=\r{H}^\r{B}_{\Gamma'}(S).$$
\end{enumerate}}
\end{theorem}
\begin{proof}
(i) Let $\epsilon>0$ and suppose that $\delta>0$ satisfies in the condition stated above. For every compact $K$ in $X$ we have
$\r{N}^\r{s}(K,\gamma_n,\epsilon)\leq\r{N}^\r{s}(K,\gamma'_n,\delta)$ and hence
$$\limsup_{n\to\infty}\frac{1}{n}\log\r{N}^\r{s}(K,\gamma_n,\epsilon)\leq\limsup_{n\to\infty}\frac{1}{n}\log\r{N}^\r{s}(K,\gamma'_n,\delta)
\leq\r{H}^\r{B}_{\Gamma'}(K).$$
This implies that $\r{H}^\r{B}_{\Gamma}(K)\leq\r{H}^\r{B}_{\Gamma'}(K)$. Hence $\r{H}^\r{B}_{\Gamma}(S)\leq\r{H}^\r{B}_{\Gamma'}(S)$.
(ii) follows from (i).
\end{proof}
Note that (i) and (ii) of Theorem \ref{2006010740} are extensions of (iv) and (v) of Theorem \ref{2005310808}.
\begin{theorem}
\emph{Let $X,\Gamma,S$ be as in Theorem \ref{2005310808}. Suppose that $S_1,\ldots,S_k$ are closed subsets of $\ov{X}$ such that
$S\subseteq\cup_{i=1}^kS_i$. Then
\begin{equation}\label{2004240850}
\r{H}^\r{B}_{\Gamma}(S)\leq\max_{i=1,\ldots,k}\r{H}^\r{B}_{\Gamma}(S_i).
\end{equation}}
\end{theorem}
\begin{proof}
Let $L,L_1,\ldots,L_k$ be compact subsets of $\ov{X}$ with $L\subseteq\cup_{i=1}^kL_i$. For $\epsilon>0$ we have
\begin{equation}\label{2004240820}
\r{N}^\r{s}(L,\gamma_n,\epsilon)\leq\sum_{i=1}^k\r{N}^\r{s}(L_i,\gamma_n,\epsilon).
\end{equation}
Let $(n_\ell)_\ell$ be a subsequence of natural numbers such that
\begin{equation}\label{2004240835}
\lim_{\ell\to\infty}\frac{1}{n_\ell}\log\r{N}^\r{s}(L,\gamma_{n_\ell},\epsilon)=
\limsup_{n\to\infty}\frac{1}{n}\log\r{N}^\r{s}(L,\gamma_n,\epsilon).
\end{equation}
Then (\ref{2004240820}) shows that there exist some $i_0$ with $1\leq i_0\leq k$ and a subsequence $(n_{\ell_j})_j$ of $(n_\ell)_\ell$ such that
$\sum_{i=1}^k\r{N}^\r{s}(L_i,\gamma_{n_{\ell_j}},\epsilon)\leq k\r{N}^\r{s}(L_{i_0},\gamma_{n_{\ell_j}},\epsilon)$.
Thus by (\ref{2004240835}) we have
\begin{equation}\label{2004240845}
\begin{split}
\limsup_{n\to\infty}\frac{1}{n}\log\r{N}^\r{s}(L,\gamma_{n},\epsilon)&\leq
\limsup_{j\to\infty}\frac{1}{n_{\ell_j}}\big(k+\log\r{N}^\r{s}(L_{i_0},\gamma_{n_{\ell_j}},\epsilon)\big)\\
&\leq\limsup_{n\to\infty}\frac{1}{n}\log\r{N}^\r{s}(L_{i_0},\gamma_{n},\epsilon)\\
&\leq\max_{i=1,\ldots,k}\limsup_{n\to\infty}\frac{1}{n}\log\r{N}^\r{s}(L_i,\gamma_{n},\epsilon)
\end{split}
\end{equation}
Letting $\epsilon\to0$ in (\ref{2004240845}) we find that
$\r{H}^\r{B}_{\Gamma}(L)\leq\max_{i=1,\ldots,k}\r{H}^\r{B}_{\Gamma}(L_i)$.
Now (\ref{2004240850}) follows from the latter inequality by letting $L_i=L\cap S_i$ where $L\subseteq S$.
\end{proof}
\begin{theorem}
\emph{Let $\Gamma=(\gamma_n)_n,\Sigma=(\sigma_n)_n$ be increasing sequences of compatible semimetrics on topological spaces $X,Y$ respectively.
Let $\gamma_n\times\sigma_n$ be the semimetric on $X\times Y$ given by
$\big((x,y),(x',y')\big)\mapsto\max\big(\gamma_n(x,x'),\sigma_n(y,y')\big)$ and denote by $\Gamma\times\Sigma$ the sequence
$(\gamma_n\times\sigma_n)_n$. Then for every $S\subseteq X$ and $T\subseteq Y$ we have
$$\r{H}^\r{B}_{\Gamma\times\Sigma}(S\times T)\leq\r{H}^\r{B}_{\Gamma}(S)+\r{H}^\r{B}_{\Sigma}(T).$$}
\end{theorem}
\begin{proof}
Let $K,L$ be compact subspaces with $K\subseteq S,L\subseteq T$. For every $n$ and $\epsilon>0$ we have
$$\r{N}^\r{s}(K\times L,\gamma_n\times\sigma_n,\epsilon)=\r{N}^\r{s}(K,\gamma_n,\epsilon)\r{N}^\r{s}(L,\sigma_n,\epsilon).$$
It follows that
\begin{equation}\label{2005310838}
\r{H}^\r{B}_{\Gamma\times\Sigma}(K\times L)\leq\r{H}^\r{B}_{\Gamma}(K)+\r{H}^\r{B}_{\Sigma}(L)\leq\r{H}^\r{B}_{\Gamma}(S)+\r{H}^\r{B}_{\Sigma}(T).
\end{equation}
Now suppose that $M\subseteq S\times T$ is a compact subspace of $X\times Y$. Then $M\subseteq\r{pr}_X(M)\times\r{pr}_Y(M)$
where $\r{pr}_X,\r{pr}_Y$ denote the canonical projections from $X\times Y$ onto $X$ and $Y$. Note that $\r{pr}_X(M)$ and $\r{pr}_Y(M)$
are compact and also $\r{pr}_X(M)\subseteq S$ and $\r{pr}_Y(M)\subseteq T$. It follows from Lemma \ref{2005260700}(v) that
$\r{H}^\r{B}_{\Gamma\times\Sigma}(M)\leq\r{H}^\r{B}_{\Gamma\times\Sigma}\big(\r{pr}_X(M)\times\r{pr}_Y(M)\big)$. Thus (\ref{2005310838})
implies that
$$\r{H}^\r{B}_{\Gamma\times\Sigma}(M)\leq\r{H}^\r{B}_{\Gamma}(S)+\r{H}^\r{B}_{\Sigma}(T).$$
The proof is complete.
\end{proof}
\section{Bowen $p$-Entropy of Subsets of Infinite Product Spaces}\label{2005240802}
The Bowen entropy \cite{Bowen1,Walters1} of continuous mappings on metric spaces is defined as follows.
Let $(X,d)$ be a metric space and let $T:X\to X$ be a continuous map. For every $n\geq1$ let $d_{n}$ be the metric on $X$ defined by
\begin{equation}\label{2006010800}
d_{n}(x,x'):=\max_{i=0,\ldots,n-1}d(T^i(x),T^i(x')).
\end{equation}
The topology induced by $d_{n}$ coincides with the original topology of $X$ and $d_{n}\leq d_{n+1}$.
Thus $\Gamma:=(d_{n})_n$ is an increasing sequence of compatible metrics on $X$. For any compact subset $K$ of $X$ the \emph{Bowen entropy of
$T$ with respect to $K$} is denoted by $\r{h}_d(T,K)$ and defined to be the value $\r{H}^\r{B}_\Gamma(K)$ as described in Section \ref{2005300733}.
The \emph{Bowen entropy of $T$} is denoted by $\r{h}_d(T)$ and defined to be the value $\r{H}^\r{B}_\Gamma(X)$.

Now we define a notion of Bowen entropy for subsets of infinite product spaces:
\begin{definition}\label{2004211415}
\emph{Let $(X,d)$ be a metric space. For every $n\geq1$  and real number $p\geq1$ let
$$d_{n,p}\big((x_k)_k,(x'_k)_k\big):=\Bigg[\sum_{i=0}^{n-1}\big(d(x_i,x'_i)\big)^p\Bigg]^\frac{1}{p}
\hspace{10mm}\big((x_k)_k,(x'_k)_k\in\ov{X}\big),$$
and also let
$$d_{n,\infty}\big((x_k)_k,(x'_k)_k\big):=\max\Big\{d(x_i,x'_i):i=0,\ldots,n-1\Big\}\hspace{10mm}\big((x_k)_k,(x'_k)_k\in\ov{X}\big).$$
Then for every real number $p$ with $1\leq p\leq\infty$, $\Delta_p:=(d_{n,p})_n$ is an increasing sequence of compatible semimetric on $\ov{X}$.
For any subset $S$ of $\ov{X}$ we call the value
$$\r{H}^\r{B}_{d,p}(S):=\r{H}^\r{B}_{\Delta_p}(S)\hspace{10mm}(1\leq p\leq\infty)$$
the \emph{Bowen $p$-entropy of $S$ with respect to $d$}.}
\end{definition}
Definition \ref{2004211415} is a generalization of the original definition of Bowen entropy:
\begin{proposition}\label{2006010751}
\emph{Let $T:(X,d)\to(X,d)$ be a continuous map. Then
$$\r{h}_d(T)=\r{H}^\r{B}_{d,\infty}(\ov{\r{Gr}}T)$$}
\end{proposition}
\begin{proof}
Let $\phi:X\to\ov{X}$ denote the canonical map associated to $T$ given by $x\mapsto(T^n(x))_{n\geq0}$. Then $\phi$ is a homeomorphism from $X$
onto $\ov{\r{Gr}}T$. For every $n\geq1$ by the notations as in (\ref{2006010800}) we have
$$d_{n}(x,x')=d_{n,\infty}\big(\phi(x),\phi(x')\big)\hspace{10mm}(x,x'\in X).$$
Thus with $\Gamma:=(d_n)_n$ we may identify the pair $(X,\Gamma)$ with the pair $(\ov{\r{Gr}}T,\Delta_\infty)$ under $\phi$. Then we have
$$\r{h}_d(T)=\r{H}^\r{B}_{\Gamma}(X)=\r{H}^\r{B}_{\Delta_\infty}(\ov{\r{Gr}}T)=\r{H}^\r{B}_{d,\infty}(\ov{\r{Gr}}T).$$
\end{proof}
Following Proposition \ref{2006010751} we define \emph{Bowen $p$-entropy of $T:(X,d)\to(X,d)$} by
\begin{equation}
\r{h}_{d,p}(T):=\r{H}^\r{B}_{d,p}(\ov{\r{Gr}}T)\hspace{10mm}(1\leq p\leq\infty).
\end{equation}
\begin{theorem}\label{2006010753}
\emph{Let $(X,d)$ be a metric space and $S\subseteq\ov{X}$. If $1\leq p\leq q\leq\infty$ then
$$\r{H}^\r{B}_{d,\infty}(S)\leq\r{H}^\r{B}_{d,q}(S)\leq\r{H}^\r{B}_{d,p}(S)\leq\r{H}^\r{B}_{d,1}(S).$$}
\end{theorem}
\begin{proof}
It is clear that $d_{n,\infty}\leq d_{n,p}$. Thus by Theorem \ref{2005310808}(iii) we have $\r{H}^\r{B}_{d,\infty}(S)\leq\r{H}^\r{B}_{d,p}(S)$.
Suppose that $1\leq p\leq q<\infty$. Let $\epsilon\in(0,1)$ be arbitrary and $n\geq1$ be fixed.
Suppose that $(x_k)_k,(x'_k)_k\in\ov{X}$ are such that
$$d_{n,p}\big((x_k)_k,(x'_k)_k\big)<\epsilon^\frac{q}{p}.$$
Then we have
$$\sum_{i=0}^{n-1}\big(d(x_i,x'_i)\big)^q\leq\sum_{i=0}^{n-1}\big(d(x_i,x'_i)\big)^p<\epsilon^q.$$
Hence $d_{n,q}\big((x_k)_k,(x'_k)_k\big)<\epsilon$. Now it follows from Theorem \ref{2006010740}(i) that
$\r{H}^\r{B}_{d,q}(S)\leq\r{H}^\r{B}_{d,p}(S)$.
\end{proof}
We do not yet know if the notion $\r{h}_{d,p}$ for $p\neq\infty$ is useful or if it has some advantages rather than the classical notion $\r{h}_d$.
But it is clear that in order to have a \emph{correct} notion of entropy in the case $1\leq p<\infty$ we must put some conditions on $(X,d)$
to guaranty at least the property $\r{h}_{d,p}(\r{id}_X)=0$ or equivalently (by Theorem \ref{2006010753}) $\r{h}_{d,1}(\r{id}_X)=0$.
Here is one such a condition:
\begin{theorem}
\emph{Suppose that $(X,d)$ is a metric space with the property that: For every compact subset $K$ of $X$ there exists a
constant $\alpha\geq1$ such that
\begin{equation}\label{2006010759}
\limsup_{\epsilon\to0}\epsilon^\alpha\r{N}^\r{s}(K,d,\epsilon)<\infty.
\end{equation}
Then $\r{h}_{d,1}(\r{id}_X)=0$.}
\end{theorem}
\begin{proof}
First of all observe that we must show that for any compact $K\subseteq X$,
$$\lim_{\epsilon\to0}\limsup_{n\to\infty}\frac{1}{n}\log\r{N}^\r{s}(K,nd,\epsilon)=0.$$
Let $\epsilon>0$ be arbitrary and fixed. There is a bound $M>0$ such that for large enough $n$,
$$\r{N}^\r{s}(K,d,\frac{\epsilon}{n})\leq M\frac{n^\alpha}{\epsilon^\alpha}.$$
Thus we have
\begin{equation*}
\begin{split}
\limsup_{n\to\infty}\frac{1}{n}\log\r{N}^\r{s}(K,nd,\epsilon)&=\limsup_{n\to\infty}\frac{1}{n}\log\r{N}^\r{s}(K,d,\frac{\epsilon}{n})\\
&\leq\limsup_{n\to\infty}\frac{\alpha\log n+\log M-\alpha\log\epsilon}{n}\\
&=0.
\end{split}
\end{equation*}
The proof is complete.
\end{proof}
We remark that the quantities similar to the one given by $\limsup$ in (\ref{2006010759}) appear in Dimension Theory of metric spaces.
The reader may easily convince him/herself that in the case that $X$ is the Euclidean space $\bb{R}^q$ and $d$ is the usual Euclidean metric 
(or any other strongly uniform equivalent metric), (\ref{2006010759}) is satisfied for any compact subset $K$ of $\bb{R}^q$ with $\alpha=q$.

The following result is a generalization of \cite[Theorem 7.4]{Walters1}.
\begin{theorem}\label{2004240800}
\emph{Let $X$ be a space with two uniformly equivalent metrics $d,d'$, i.e. for every $\epsilon>0$ there exists $\delta>0$ such that
if $d(x,x')<\delta$ then $d'(x,x')<\epsilon$, and such that if $d'(x,x')<\delta$ then $d(x,x')<\epsilon$. Then for every $S\subseteq\ov{X}$
we have
$$\r{H}^\r{B}_{d,\infty}(S)=\r{H}^\r{B}_{d',\infty}(S).$$}
\end{theorem}
\begin{proof}
It is easily seen that $\Delta_{\infty}=(d_{n,\infty})_n$ and $\Delta'_{\infty}=(d'_{n,\infty})_n$ are equi uniformly-equivalent.
Thus the desired result follows from Theorem \ref{2006010740}.
\end{proof}
We have the following trivial corollary of Theorem \ref{2004240800}.
\begin{corollary}\label{2006010844}
\emph{If $X$ is compact then $\r{H}^\r{B}_{d,\infty}(S)$ only depends on $S$ and the topology of $d$.}
\end{corollary}
\begin{theorem}
\emph{Let $X$ be a space with two strongly uniform equivalent metrics $d,d'$, i.e. there exist constants $c,c'>0$ such that
$c'd'\leq d\leq cd'$. Then for every $S\subseteq\ov{X}$
we have
$$\r{H}^\r{B}_{d,p}(S)=\r{H}^\r{B}_{d',p}(S)\hspace{10mm}(1\leq p\leq\infty).$$}
\end{theorem}
\begin{proof}
It is easily verified that $\Delta_{p}=(d_{n,p})_n$ and $\Delta'_{p}=(d'_{n,p})_n$ are equi strongly-uniform-equivalent.
Thus the desired result follows from Theorem \ref{2005310808}(v).
\end{proof}
Now we show that as in the classical case \cite[Theorem 7.8]{Walters1} Bowen $\infty$-entropy and generalized topological entropy
coincide in compact metric spaces.
\begin{theorem}\label{2004270841}
\emph{Let $(X,d)$ be a compact metric space and $L\subseteq\ov{X}$ be closed. Then
$$\r{H}^\r{B}_{d,\infty}(L)=\r{H}^\r{t}_{X}(L).$$}
\end{theorem}
\begin{proof}
Let $(\c{U}_k)_k$ be a sequence of open covers of $X$ such that $\r{diam}(\c{U}_k)<\frac{1}{k}$ for every $k$.
It follows from (ii), (iii) and (iv) of Lemma \ref{2005260700} that
$$\r{N}_{L/\ov{X}}(\ov{\c{U}_k}^n)\leq\r{N}^\r{s}(L,d_{n,\infty},\frac{1}{2k})\leq\r{N}_{L/\ov{X}}(\ov{\c{U}_{2k}}^n).$$
This implies that
\begin{equation*}\label{2004270840}
\r{H}^\r{t}_{X}(L,\c{U}_k)\leq\limsup_{n\to\infty}\frac{1}{n}\log\r{N}^\r{s}(L,d_{n,\infty},\frac{1}{2k})\leq\r{H}^\r{t}_{X}(L,\c{U}_{2k})
\end{equation*}
Now let $k\to\infty$. Then the desired result is concluded from Theorem \ref{2004270830}.
\end{proof}
Note that the above result contains the result stated in Corollary \ref{2006010844} for closed subsets.
\begin{theorem}
\emph{Let $(X,d)$ be a finite metric space. Then for any $S\subseteq\ov{X}$ and every $p\geq1$,
$$\r{H}^\r{B}_{d,p}(S)=\r{H}^\r{B}_{d,\infty}(S).$$}
\end{theorem}
\begin{proof}
It follows from the obvious fact that if $\epsilon<\min\big\{d(x,x'):x,x'\in X, x\neq x'\big\}$ then for any (compact) subset $K$ of $\ov{X}$,
every $n\geq1$, and every $p\geq1$ we have
$$\r{N}^\r{s}(K,d_{n,p},\epsilon)=\r{N}^\r{s}(K,d_{n,\infty},\epsilon).$$
\end{proof}
\end{document}